\documentclass[12pt]{amsart}
\usepackage{amsmath,amssymb,amsbsy,amsfonts,latexsym,amsopn,amstext,cite,
                                               amsxtra,euscript,amscd,bm}
\usepackage{url}

\usepackage{mathrsfs}

\usepackage{color}
\usepackage[colorlinks,linkcolor=blue,anchorcolor=blue,citecolor=blue,backref=page]{hyperref}
\usepackage{color}
\usepackage{graphics,epsfig}
\usepackage{graphicx}
\usepackage{float}
\usepackage{epstopdf}
\hypersetup{breaklinks=true}

\usepackage[np]{numprint}
\npdecimalsign{\ensuremath{.}}

\usepackage{bibentry}

\usepackage[english]{babel}
\usepackage{mathtools}
\usepackage{todonotes}
\usepackage{url}
\usepackage[colorlinks,linkcolor=blue,anchorcolor=blue,citecolor=blue,backref=page]{hyperref}

\usepackage[norefs,nocites]{refcheck}

\begin{document}

\newcommand{\bs}{\boldsymbol}
\def \a{\alpha} \def \b{\beta} \def \d{\delta} \def \e{\varepsilon} \def \g{\gamma} \def \k{\kappa} \def \l{\lambda} \def \s{\sigma} \def \t{\theta} \def \z{\zeta}

\newcommand{\mb}{\mathbb}

\newtheorem{theorem}{Theorem}
\newtheorem{lemma}[theorem]{Lemma}
\newtheorem{claim}[theorem]{Claim}
\newtheorem{cor}[theorem]{Corollary}
\newtheorem{conj}[theorem]{Conjecture}
\newtheorem{prop}[theorem]{Proposition}
\newtheorem{definition}[theorem]{Definition}
\newtheorem{question}[theorem]{Question}
\newtheorem{example}[theorem]{Example}
\newcommand{\hh}{{{\mathrm h}}}
\newtheorem{remark}[theorem]{Remark}

\numberwithin{equation}{section}
\numberwithin{theorem}{section}
\numberwithin{table}{section}
\numberwithin{figure}{section}

\def\sssum{\mathop{\sum\!\sum\!\sum}}
\def\ssum{\mathop{\sum\ldots \sum}}
\def\iint{\mathop{\int\ldots \int}}

\newcommand{\diam}{\operatorname{diam}}

\def\squareforqed{\hbox{\rlap{$\sqcap$}$\sqcup$}}
\def\qed{\ifmmode\squareforqed\else{\unskip\nobreak\hfil
\penalty50\hskip1em \nobreak\hfil\squareforqed
\parfillskip=0pt\finalhyphendemerits=0\endgraf}\fi}

\newfont{\teneufm}{eufm10}
\newfont{\seveneufm}{eufm7}
\newfont{\fiveeufm}{eufm5}
%
%
\newfam\eufmfam
     \textfont\eufmfam=\teneufm
\scriptfont\eufmfam=\seveneufm
     \scriptscriptfont\eufmfam=\fiveeufm
%
%
\def\frak#1{{\fam\eufmfam\relax#1}}

\newcommand{\bflambda}{{\boldsymbol{\lambda}}}
\newcommand{\bfmu}{{\boldsymbol{\mu}}}
\newcommand{\bfxi}{{\boldsymbol{\eta}}}
\newcommand{\bfrho}{{\boldsymbol{\rho}}}

\def\eps{\varepsilon}

\def\fK{\mathfrak K}
\def\fT{\mathfrak{T}}
\def\fL{\mathfrak L}
\def\fR{\mathfrak R}

\def\fA{{\mathfrak A}}
\def\fB{{\mathfrak B}}
\def\fC{{\mathfrak C}}
\def\fM{{\mathfrak M}}
\def\fS{{\mathfrak  S}}
\def\fU{{\mathfrak U}}

\def\sssum{\mathop{\sum\!\sum\!\sum}}
\def\ssum{\mathop{\sum\ldots \sum}}
\def\dsum{\mathop{\quad \sum \qquad \sum}}
\def\iint{\mathop{\int\ldots \int}}
 
\def\T {\mathsf {T}}
\def\Tor{\mathsf{T}_d}
\def\Tore{\widetilde{\mathrm{T}}_{d} }

\def\sM {\mathsf {M}}
\def\sL {\mathsf {L}}
\def\sK {\mathsf {K}}
\def\sP {\mathsf {P}}

\def\ss{\mathsf {s}}

\def\Kmnd{\cK_d(m,n)}
\def\Kmnp{\cK_p(m,n)}
\def\Kmnq{\cK_q(m,n)}

\def \balpha{\bm{\alpha}}
\def \bbeta{\bm{\beta}}
\def \bgamma{\bm{\gamma}}
\def \bdelta{\bm{\delta}}
\def \bzeta{\bm{\zeta}}
\def \blambda{\bm{\lambda}}
\def \bchi{\bm{\chi}}
\def \bphi{\bm{\varphi}}
\def \bpsi{\bm{\psi}}
\def \bxi{\bm{\xi}}
\def \bnu{\bm{\nu}}
\def \bomega{\bm{\omega}}

\def \bell{\bm{\ell}}

\def\eqref#1{(\ref{#1})}

\def\vec#1{\mathbf{#1}}

\newcommand{\abs}[1]{\left| #1 \right|}

\def\Zq{\mathbb{Z}_q}
\def\Zqx{\mathbb{Z}_q^*}
\def\Zd{\mathbb{Z}_d}
\def\Zdx{\mathbb{Z}_d^*}
\def\Zf{\mathbb{Z}_f}
\def\Zfx{\mathbb{Z}_f^*}
\def\Zp{\mathbb{Z}_p}
\def\Zpx{\mathbb{Z}_p^*}
\def\cM{\mathcal M}
\def\cE{\mathcal E}
\def\cH{\mathcal H}

\def\le{\leqslant}

\def\ge{\geqslant}

\def\sfB{\mathsf {B}}
\def\sfC{\mathsf {C}}
\def\sfS{\mathsf {S}}
\def\sfI{\mathsf {I}}
\def\L{\mathsf {L}}
\def\FF{\mathsf {F}}

\def\sE {\mathscr{E}}
\def\sS {\mathscr{S}}

\def\cA{{\mathcal A}}
\def\cB{{\mathcal B}}
\def\cC{{\mathcal C}}
\def\cD{{\mathcal D}}
\def\cE{{\mathcal E}}
\def\cF{{\mathcal F}}
\def\cG{{\mathcal G}}
\def\cH{{\mathcal H}}
\def\cI{{\mathcal I}}
\def\cJ{{\mathcal J}}
\def\cK{{\mathcal K}}
\def\cL{{\mathcal L}}
\def\cM{{\mathcal M}}
\def\cN{{\mathcal N}}
\def\cO{{\mathcal O}}
\def\cP{{\mathcal P}}
\def\cQ{{\mathcal Q}}
\def\cR{{\mathcal R}}
\def\cS{{\mathcal S}}
\def\cT{{\mathcal T}}
\def\cU{{\mathcal U}}
\def\cV{{\mathcal V}}
\def\cW{{\mathcal W}}
\def\cX{{\mathcal X}}
\def\cY{{\mathcal Y}}
\def\cZ{{\mathcal Z}}
\newcommand{\rmod}[1]{\: \mbox{mod} \: #1}

\def\cg{{\mathcal g}}

\def\vy{\mathbf y}
\def\vr{\mathbf r}
\def\vx{\mathbf x}
\def\va{\mathbf a}
\def\vb{\mathbf b}
\def\vc{\mathbf c}
\def\ve{\mathbf e}
\def\vf{\mathbf f}
\def\vg{\mathbf g}
\def\vh{\mathbf h}
\def\vk{\mathbf k}
\def\vm{\mathbf m}
\def\vz{\mathbf z}
\def\vu{\mathbf u}
\def\vv{\mathbf v}

\def\e{{\mathbf{\,e}}}
\def\ep{{\mathbf{\,e}}_p}
\def\eq{{\mathbf{\,e}}_q}

\def\Tr{{\mathrm{Tr}}}
\def\Nm{{\mathrm{Nm}}}

 \def\SS{{\mathbf{S}}}

\def\lcm{{\mathrm{lcm}}}

 \def\0{{\mathbf{0}}}

\def\({\left(}
\def\){\right)}
\def\l|{\left|}
\def\r|{\right|}
\def\fl#1{\left\lfloor#1\right\rfloor}
\def\rf#1{\left\lceil#1\right\rceil}
\def\sumstar#1{\mathop{\sum\vphantom|^{\!\!*}\,}_{#1}}

\def\mand{\qquad \mbox{and} \qquad}

\def\tblue#1{\begin{color}{blue}{{#1}}\end{color}}




\hyphenation{re-pub-lished}

\mathsurround=1pt

\def\bfdefault{b}

\def \F{{\mathbb F}}
\def \K{{\mathbb K}}
\def \N{{\mathbb N}}
\def \Z{{\mathbb Z}}
\def \P{{\mathbb P}}
\def \Q{{\mathbb Q}}
\def \R{{\mathbb R}}
\def \C{{\mathbb C}}
\def\Fp{\F_p}
\def \fp{\Fp^*}

 \def \xbar{\overline x}

\title[Norms of  Maximal Operators on Weyl Sums]{Bounds on the  Norms of  Maximal Operators on Weyl Sums}

\author[R. C. Baker] {Roger C.~Baker} 
\address{Department of Mathematics, Brigham Young University, 
Provo, UT 84602, USA} 
\email{baker@math.byu.edu}

 \author[C. Chen] {Changhao Chen}
\address{School of Mathematical Sciences, Anhui University, Hefei 230601, China}
\email{chench@ahu.edu.cn}

 \author[I. E. Shparlinski] {Igor E. Shparlinski}
\address{Department of Pure Mathematics, University of New South Wales,
Sydney, NSW 2052, Australia}
\email{igor.shparlinski@unsw.edu.au}

\begin{abstract}   We obtain new estimates on the {\it maximal 
operator} applied to the Weyl sums.  We also consider the quadratic case (that is, Gauss sums) in more details. In wide ranges of  parameters our estimates are optimal and match lower bounds. 
Our approach is based on a
combination of ideas of Baker (2021) and Chen and Shparlinski (2020). 
\end{abstract}


\maketitle

%

\section{Introduction}

\subsection{Set-up and motivation} 
Given a family $\bphi = \(\varphi_1, \ldots, \varphi_d\)\in \Z[T]^d$  of $d$ distinct 
nonconstant polynomials, a positive integer $k\le d$ and a real positive parameter $\rho$,  we consider the $L^\rho$-norms  of the so called {\it maximal operator } 
 \begin{align*}
\sM_{k, \rho} (\bphi, N) & = \left\| \sup_{\vy \in  \T^{d-k}}
\left| S_{\bphi}( \vx, \vy; N)  \right|\right\|_{\L^\rho\(\T_k\)}\\
& =\(\int_{\T_k} \sup_{\vy \in\T_{d-k}}
\left| S_{\bphi}( \vx, \vy; N)  \right|^\rho d\vx\)^{1/\rho} 
 \end{align*}
on the {\it Weyl sums\/} 
$$
S_{\bphi}( \vx, \vy; N) = \sum_{n=1}^{N} \e\(\sum_{j=1}^k x_j \varphi_j(n)+ \sum_{j=1}^{d-k}y_j\varphi_{k+j}(n)\), 
$$
where  $\e(z) = \exp(2\pi i z)$  with two groups coefficients 
$$\vx = (x_1, \ldots, x_k)\in \T_k \mand \vy =(y_1, \ldots, y_{d-k})\in \T_{d-k}, 
$$ 
where   
$$
\T_\nu = [0,1]^\nu
$$
is the $\nu$-dimensional unit cube.   

Such bounds, as well as bounds on  $\sup_{\vy \in \T_{d-k}}
\left| S_{\bphi}( \vx, \vy; N)  \right|$ which hold for almost all $\vx \in \T_k$
have recently been considered in a number of works~\cite{ACHK, AKP, Bak2, Barr, BPPSV, BrSh, ErdSha,  ChSh1, ChSh2, ChSh4, Wool}. Results of this kind add to our understanding of Weyl sums. Besides the  interest to these results is ignited by 
applications outside of number theory, see~\cite{ACP, ACHK, AKP,  BPPSV, ErdSha, Pierce}.

Here, to exhibit our idea in the clearest possible form, we consider the 
special, but perhaps most interesting, case  when 
 \begin{equation}\label{eq:set phi}
\{\varphi_1(T), \ldots, \varphi_d(T)\} = \{T, \ldots, T^d\}.
 \end{equation}
 We emphasise that in~\eqref{eq:set phi} we request the {\it equality of sets\/} rather 
 than of sequences. Thus~\eqref{eq:set phi} means that  
$$\varphi_i(T) = T^{\pi(i)}, \qquad i =1, \ldots, d,
$$ 
for some permutation $\pi \in \cS_d$. 

We also note that most of our results depend on the following parameters 
\begin{equation} 
\label{eq:sigmak}
\tau_k(\bphi)=\sum_{j=1}^k \deg \varphi_j   \mand  \sigma_k(\bphi)=\sum_{j=k+1}^d \deg \varphi_j.
\end{equation} 
It is also convenient to define 
\begin{equation}
\label{eq:sd}
s(d) = \frac{d(d+1)}{2}. 
\end{equation}

In the case $k=0$ we trivially have $\sM_{0, \rho} (\bphi, N)= N.$
We also observe that the case of $k = d$ corresponds to the {\it Vinogradov Mean Value Theorem\/}, 
which has recently  been obtained in an optimal form
 by Bourgain,  Demeter and Guth~\cite{BDG} 
and Wooley~\cite{Wool2,Wool3}.  Hence we are mostly interested in the case $1 \le k < d$.

\subsection{Notation} 

Throughout the paper, the notation $U = O(V)$, 
$U \ll V$ and $ V\gg U$  are equivalent to $|U|\leqslant c V$ for some positive constant $c$, 
which throughout the paper may depend on the degree $d$ and occasionally on the small real positive 
parameter $\varepsilon$ and the arbitrary real parameter $t$. 

For any quantity $V> 1$ we write $U = V^{o(1)}$ (as $V \rightarrow \infty$) to indicate a function of $V$ which 
satisfies $ V^{-\eps} \le |U| \le V^\eps$ for any $\eps> 0$, provided $V$ is large enough. One additional advantage 
of using $V^{o(1)}$ is that it absorbs $\log V$ and other similar quantities without changing  the whole 
expression.  
 
\subsection{Previous results} 

Here we give a brief outline of previously known upper and lower bounds on $\sM_{k, \rho} (\bphi, N)$.
We recall our definitions~\eqref{eq:sigmak} and~\eqref{eq:sd}. 

The first result in this direction has been given by Chen and Shparlinski~\cite[Theorem~2.1]{ChSh1}, which  implies that  for any real positive  $\rho \le 2s(d)+d-k$,   
\begin{equation}
\label{eq:Bound CS} 
\sM_{k, \rho} (\bphi, N)  \le N^{\mu(k, \bphi)+ o(1)}, 
\end{equation}  
where 
$$
\mu(k, \bphi)= \frac{s(d)+\sigma_k(\bphi)+d-k } { 2 s(d) +d -k} = 1 -  \frac{\tau_k(\bphi)} { 2 s(d) +d -k} .
$$ 
It is also shown in~\cite[Theorem~2.1]{ChSh1} that in 
the special cases of the functions 
$$
\bphi_{2,1}(T) = ( T^2,T)
$$
we have 
\begin{equation}
\label{eq:Bound CS Gauss} 
N^{1/2}  \le \sM_{1, 2} \(\bphi_{2,1}, N\)  \ll \(N \log N\)^{1/2} , 
\end{equation}
while for 
$$
\bphi_{3,2,1}(T) = (T^3, T^2,T)
$$
we have 
$$
2^{1/4} N^{1/2}  + O\(N^{-1/2}\) \le \sM_{1, 4} \(\bphi_{3,2,1}, N\)  
\ll N^{3/4} \( \log N\)^{1/4}.
$$
Barron~\cite{Barr} has recently  obtained the following estimate 
\begin{equation}
\label{eq:Bound Barr}
\sM_{1,4} \(\bphi_{1,2}, N\)  \le N^{3/4+o(1)} 
\end{equation}
for 
$$
\bphi_{1,2}(T) = ( T,T^2). 
$$
Using a different  approach,  Baker~\cite{Bak2} has   refined~\eqref{eq:Bound Barr}  with matching  upper and lower bounds: 
\begin{equation}
\label{eq:Bound BakM}
N^{a(\rho)}  (\log N)^{b(\rho)} \ll 
\sM_{1,\rho} \(\bphi_{1,2}, N\)  \ll 
N^{a(\rho)} (\log N)^{b(\rho)}, 
\end{equation} 
where  
$$
a(\rho) = \begin{cases}
 3/4 & \text{for} \  1\le \rho \le 4,\\
1-1/\rho  & \text{for} \ \rho> 4, 
\end{cases} \qquad 
b(\rho) = \begin{cases}
1/4 & \text{for} \  \rho = 4,\\
0 & \text{for} \ \rho \ge 1,\  \rho \ne 4.
\end{cases}
$$  
Here we also consider the quadratic case, that is,  the case of Gauss sums,  in more detail.
It is convenient to introduce the notations 
$$
G(x,y; N) = \sum_{n=1}^{N} \e\(xn + yn^2\)
$$
and
 \begin{align*}
\sK_{\rho} (N) & = \left\| \sup_{y \in \T }
\left| G(x,y; N) \right|\right\|_{\L^\rho\(\T\)}  =\(\int_\T \sup_{y \in  \T }
\left| G(x,y; N) \right|^\rho d x\)^{1/\rho}, \\
\sL_{\rho} (N) & = \left\| \sup_{x \in  \T }
\left| G(x,y; N) \right|\right\|_{\L^\rho\( \T\)}  =\(\int_\T \sup_{x \in  \T }
\left| G(x,y; N) \right|^\rho d y\)^{1/\rho} , 
 \end{align*}
 where 
 $$
 \T = \T_1 = [0,1].
$$
 In particular, the bound~\eqref{eq:Bound CS Gauss}  can now be written as 
 $$
 N^{1/2}  \le \sL_{2} \(N\)  \ll \(N \log N\)^{1/2} 
$$
 and   the bound~\eqref{eq:Bound BakM} as
$$
N^{a(\rho)}  (\log N)^{b(\rho)} \ll 
\sK_{\rho} (N)  \ll  
N^{a(\rho)} (\log N)^{b(\rho)}. 
$$
We also consider norms of maximal operators along other straight lines.

\subsection{New bounds} 

Here we combine the ideas from~\cite{Bak2} and~\cite{ChSh1}
and obtain  new bounds which improve~\eqref{eq:Bound CS}.

For large $\rho$ we have the following upper and lower bounds of the same order of magnitude.

\begin{theorem}
\label{thm:general-large rho} 
Suppose that $\bphi \in \Z[T]^d$ satisfies~\eqref{eq:set phi}. 
For any real positive $\rho \ge 2 s(d) +d -k$ 
we have 
$$
  N^{1-\tau_k(\bphi)/\rho} \ll \sM_{k,\rho} (\va, \bphi, N) \le N^{1-\tau_k(\bphi)/\rho+o(1)}, \quad N\rightarrow \infty.
$$ 
\end{theorem}

Note that by the convexity (that is, by the H\"older inequality),  Theorem~\ref{thm:general-large rho},  taken with $\rho = 2 s(d) +d -k$, 
implies~\eqref{eq:Bound CS}. 

Our next result gives better bounds  for small values of $d$, namely for $3 \le d \le 6$,  and  some choices of other parameters.

\begin{theorem} \label{thm:smaller rho}
Suppose that $d \ge 3$ and $\bphi \in \Z[T]^d$ satisfies~\eqref{eq:set phi}. 
For any real $\rho >0$ 
we have 
$$
\sM_{k,\rho} (\va, \bphi, N) \le N^{1-1/2^{d-1}+o(1)}+N^{1-\tau_k(\bphi)/\rho+o(1)}, \quad N\rightarrow \infty.
$$ 
\end{theorem}

Elementary computation shows that  
 Theorem~\ref{thm:smaller rho} provides a better bound than~\eqref{eq:Bound CS}  
 in the range $\rho <  2^{d-1} \tau_k(\bphi)$ provided
\begin{equation}
\label{eq:new bound}
2^{d-1} \tau_k(\bphi) <2s(d)+d-k.
\end{equation}

For large $d$ this condition is never satisfied, however for each $d \in \{3,4,5,6\}$ we give 
examples of parameters when  Theorem~\ref{thm:smaller rho} improves~\eqref{eq:Bound CS}. 

For each $1\le k<d$ denote 
$$
\bphi(T)\vert_k=\{\varphi_1(T), \ldots, \varphi_k(T)\}. 
$$

We now list all the possible choices of $d, k$ and $\bphi(T)\vert_k$ with $\bphi(T)$ as in~\eqref{eq:set phi} such that Theorem~\ref{thm:smaller rho} gives better bounds than \eqref{eq:Bound CS}.

For $d=3$, the condition~\eqref{eq:new bound} becomes 
$\tau_k(\bphi)<(15-k)/4$. Thus for $k=1$ it is sufficient to have  
$\tau_1(\bphi) <7/2$, which holds for all the possible choices  that 
$\bphi(T)\vert_1\in \{T, T^2, T^3\}$. For $k=2$, it is sufficient to have  
$\tau_2(\bphi) <13/4$. Since $\tau_2(\bphi) \neq 1, 2 $, we obtain that $\tau_2(\bphi) =3$. Thus we have only one choice for $\bphi(T)\vert_2$, that is $\bphi(T)\vert_2=\{T, T^2\}$.

For $d=4$, the condition~\eqref{eq:new bound} becomes 
$
\tau_k(\bphi)<(24-k)/8.
$
Thus for $k=1$ it is sufficient to have $\tau_k(\bphi)=1, 2$, which means that $\bphi(T)\vert_1\in \{T, T^2\}$.  For $k=2$, then it is sufficient to have 
$\tau_2(\bphi) =1, 2$. But by our definition of $\bphi$ we can not have $\tau_2(\bphi) =1, 2$.

For $d=5$, the condition \eqref{eq:new bound} becomes 
$
\tau_k(\bphi)<(35-k)/16.
$
For $k=1$ it is sufficient to have $\tau_k(\bphi)=1, 2$, which implies that $\bphi(T)|_1\in \{T, T^2\}$. 
Moreover, for other values $k$ we do not have new bounds for these cases.

For $d=6$, the condition \eqref{eq:new bound} becomes 
$
\tau_k(\bphi)<(48-k)/32.
$
For $k=1$ it is sufficient to take $\tau_1(\bphi) =1$ which implies that $\bphi(T)|_1=T$.
However, for other choices of $k$ Theorem~\ref{thm:smaller rho} does not yield a new bound. 

We remark that for any $d\ge 7$ and $1\le k<d$, and any $\bphi(T)\vert_k$ the bound \eqref{eq:Bound CS} gives 
a better upper bound than Theorem~\ref{thm:smaller rho}.

For the maximal operators on Gauss sums we have the following result.

\begin{theorem}
\label{thm:Gauss L rho} 
For any real  $\rho>0$ 
we have 
$$
\sL_{\rho} (N) \le   N^{1/2+o(1)}+N^{1-2/\rho+o(1)}, \quad N\rightarrow \infty. 
$$ 
\end{theorem}

Note that if $\rho \ge 4$ then the second term  in the bound of Theorem~\ref{thm:Gauss L rho} dominates and 
it becomes a special case of the optimal upper bound of Theorem~\ref{thm:general-large rho}. We also note that for any
$y\in \T$ we have  
$$
\int_\T \left|\sum_{n=1}^N\e(xn+yn^2)\right|^2dx=N,
$$
thus 
$
\sup_{x\in \T} |G(x, y)|\ge N^{1/2},
$
and hence 
$$
\int_\T  \sup_{x\in \T}\left| G(x, y)\right|^\rho dy\ge N^{\rho/2}.
$$
Therefore, we conclude that Theorem~\ref{thm:Gauss L rho} is  optimal for any $\rho >0$.

From Theorem~\ref{thm:Gauss L rho} we derive the following bounds for the mean values of short sums of Gauss sums, which improves the bounds~\cite[Corollary~2.2]{ChSh1} for this setting.  For $M\in \Z$, we consider Gauss sums over short intervals, that is, 
$$
G(x, y; M, N)= \sum_{n=M+1}^{M+N} \e(xn+yn^2).
$$
Elementary arithmetic shows that 
$$
|G(x, y; M, N)|=|G(x+2My, y; N)|,
$$
and hence 
$$
\sup_{M\in \Z}|G(x, y; M, N)|\le \sup_{u\in \T}|G(u, y; N)|.
$$

\begin{cor}
\label{cor:short}
Using Theorem~\ref{thm:Gauss L rho} and  above notation, for any $\rho>0$ we have 
$$
\int_{\T_2}\sup_{M\in \Z}|G(x, y; M, N)|^{\rho}dxdy\le N^{\rho/2+o(1)}+N^{1-2/\rho+o(1)}.
$$
\end{cor}

For any $t\in \R$ the projection $\pi_{t}: \R^2\rightarrow \R$ is defined as 
$$
\pi_t(x, y)=x+ty.
$$ 
Let  
$$
\sP_{t,\rho} (N)      =\(\int_{\R}  \sup_{(x, y)\in \pi_{t}^{-1}(z)\cap \T_2 } \left| G(x,y; N) \right|^\rho d z\)^{1/\rho}.
$$

A similar  quantity for much more general sums has been treated in~\cite[Theorem~2.3]{ChSh1}. In the special case of 
Gauss sums we obtain a stronger result. 

\begin{theorem}
\label{thm:any-t}
For any $t\in \R$ and any $\rho>0$ we have 
$$
\sP_{t,\rho} (N)   \le  N^{ 5/6+o(1)}+N^{1-1/\rho+o(1)}, \quad N\rightarrow \infty.
$$
\end{theorem}

We remark that  Theorem~\ref{thm:any-t} improves the bound  $N^{6/7+o(1)}$ for $\rho \le 7$ 
which follows from the general estimate of~\cite[Theorem~2.3]{ChSh1}.

For a rational number $t$ we have the following better bounds.

\begin{theorem}
\label{thm:rational-t}
For any $t\in \Q$ and any  $\rho>0$ we have 
$$
\sP_{t,\rho} (N)  \le 
 N^{3  /4+o(1)}  + N^{1-1/\rho+o(1)} , \quad N\rightarrow \infty.
 $$
\end{theorem} 

We observe that for $t = 0$ we have   
$$
\sP_{0,\rho} (N)   = \sM_{1,\rho} \(\bphi_{1,2}, N\)  = \sK_{\rho} (N).
$$

Hence Theorem~\ref{thm:rational-t}   implies a slightly less precise version of the upper bound in~\eqref{eq:Bound BakM}.

\section{Bounds of Weyl sums}

\subsection{Approximations of Weyl sums} 
We recall a result of Vaughan~\cite[Theorem~7.2]{Vau}
approximating general Weyl sums
$$
\sfS_{d}(\vu; N)  = \sum_{n=1}^{N} \e\(u_1n + \ldots + u_dn^d\)
$$
 by complete rational sums. 

\begin{lemma} 
\label{lem:Weyl-Approx} Suppose  that for 
integers $q, r_1, \ldots, r_d$ with 
$$
\gcd(q, r_1, \ldots, r_d) = 1
$$
 we have 
$$
\left| u_j - \frac{r_j}{q} \right | \le  \xi_j, 
\qquad j =1, \ldots, d, 
$$
for some $ \xi_1, \ldots, \xi_d\in \R$. Then 
$$
\sfS_{d}(\vu; N) = q^{-1} \sfS_{d}\(q^{-1} \vr; q\) \sfI\(\bxi\) + \Delta, 
$$
where $\vr = \( r_1, \ldots, r_d\)$,  $\bxi=\(\xi_1, \ldots, \xi_d\)$, 
$$
\sfI\(\bxi\) = \int_0^N \e\(\xi_1 z +\ldots +\xi_dz^d\) d z
$$
and $\Delta$ satisfies the bound
$$
\Delta \ll q \(1 + |\xi_1|N + \ldots + |\xi_d| N^d\). 
$$
\end{lemma}

To apply Lemma~\ref{lem:Weyl-Approx} we need to recall well-known bounds for the complete rational sum $ \sfS_{d}\(q^{-1} \vr; q\)$ and the oscillating 
integral $\sfI\(\bxi\)$. 

For $\sfI\(\bxi\)$  by~\cite[Theorem~7.3]{Vau} we have 
\begin{equation}
\label{eq:Int(xi)}
\sfI(\bxi)\ll N\min_{j=1, \ldots, d}\{1, \xi_j^{-1/d}N^{-j/d}\}.
\end{equation}
Hence,  we now concentrate on the sums $ \sfS_{d}\(q^{-1} \vr; q\)$. 

\subsection{Bounds of complete rational sums}

We denote 
\begin{equation}
\label{eq:comp sum} 
S_{d,q}(\vb) =  \sfS_{d}\(q^{-1} \vb; q\)=\sum_{n=1}^{q}\e_q(b_1n+\ldots+b_dn^d), 
\end{equation}
where $\e_q(k) = \exp(2 \pi k/q)$.

If $p$ is a prime number, then we have the classical Weil bound,  see, for example,~\cite[Theorem~5.38]{LN}.

\begin{lemma} 
\label{lem:Weil}
For a prime $p$, and $\vb\in \Z^d$ with 
$$
\gcd(p, b_1, \ldots, b_d) = 1, 
$$ 
we have $S_{d,p}(\vb)\le  (d-1)\sqrt{p}$.
\end{lemma}

If $q=p^m$ for some prime number $p$ and integer $m \ge 2$ then we have the following bound, see, for example,~\cite[Equation~(2.5)]{CPR}.  

\begin{lemma} 
\label{lem:Exp pm}
For a prime $p$, an integer $m\ge 1$ and $\vb\in \Z^d$ with 
$$
\gcd(p, b_1, \ldots, b_d) = 1, 
$$ 
we have $S_{d,p^m} (\vb)\le (d-1) p^{m-1}$.
\end{lemma}  

The following estimate  (see, for example,~\cite{DiQi,Stech}) is a slight improvement of the 
classical bound which contains an additional factor $q^{o(1)}$, see~\cite[Theorem~7.1]{Vau}.
It makes our calculation slightly less cluttered (but is not necessary for our final result). 

\begin{lemma} 
\label{lem:Hua} For an integer  $q\ge 1$  and  $\vb\in \Z^d$ with 
$$
\gcd(q, b_1, \ldots, b_d) = 1, 
$$ 
we have 
$$
S_{d,q}(\vb) \ll q^{1-1/d}.
$$ 
\end{lemma}

An integer number $n$ is called
\begin{itemize}
\item  {\it $r$-th power free\/} if   any prime number $p \mid n$ satisfies $p^r \nmid n$; 

\item  {\it  $r$-th power  full\/} if any prime number  $p \mid n$ satisfies $p^r \mid n$. 
\end{itemize}
We note that $1$ is both $r$-th power  free and $r$-th power full for any $r \in \N$. 

Our main tool is the following bound on $|S_{d,q}(\vb)|$.

\begin{lemma}
\label{lem:d-power-factor}
Write  an integer $q\ge 1$ as 
$q=q_2 \ldots q_d$ with $\gcd(q_i,q_j) = 1$, $2 \le i < j \le d$, such that
\begin{itemize} 
\item $q_2\ge 1$ is  cube  free,
\item $q_i$ is $i$-th power full but $(i+1)$-th power free when $3 \le i \le d-1$,
\item $q_d$ is   $d$-th power full. 
\end{itemize}
For  $\vb\in \Z^d$ with 
$$
\gcd\(q, b_1, \ldots, b_d\) = 1, 
$$ 
we have 
$$
|S_{d,q}(\vb)| \le \prod_{i=2}^d q_i^{1-1/i} q^{o(1)}. 
$$ 
\end{lemma}

\begin{proof} We factor $q$ as in a product of distinct primes and prime squares as 
$$
q=p_{1}^{m_1} \ldots p_{s}^{m_s},
$$
It is well known that the function $S_{d,q}(\vb)$ is ``multiplicative", this follows by applying an
argument similar to~\cite[Equation~(12.21)]{IwKow} or~\cite[Lemma~2.10]{Vau}. It means that for any vector $\vb$ there exist vectors  $\ve_{j}= \( e_{1,j}, \ldots, e_{d,j}\)$ such that 
$$
\gcd\(p_j, e_{1,j}, \ldots, e_{d,j}\)   = 1, 
$$
for $i=1, \ldots, s$, and
$$
S_{d,q}(\vb)= \prod_{j=1}^{s}S_{d,p_{j}^{m_j}}(\ve_j)  . 
$$
Applying Lemmas~\ref{lem:Weil}, \ref{lem:Exp pm} and~\ref{lem:Hua}, we obtain 
$$
S_{d,p_{j}^{m_j}} (\ve_j)  \le (d-1) p_{j}^{m_j(1 -  1/\max\{2,\min\{m_j, d\}\})}, \qquad 
j =1, \ldots s.
$$
We now form 
\begin{itemize}
\item $q_2$ as  the product of powers $p_{j}^{m_j}$ with $m_j = 1,2$,
\item $q_i$, $3 \le i \le d-1$,  as  the product of powers $p_{j}^{m_j}$ with $m_j = i$,
\item $q_d$,  as  the product of powers $p_{j}^{m_j}$ with $m_j \ge d$.
\end{itemize}
and 
$$
|S_{d,q}(\vb)|   
\le  (d-1)^{s}  \prod_{i=2}^d q_i^{1-1/i}.
$$
Since obviously 
$$
s! \le  \prod_{i=1}^{s}p_{i} \le q, 
$$
we see that $(d-1)^{s} = q^{o(1)}$,  which finishes the proof.
\end{proof}

\subsection{Structure of large Weyl sums} 
\label{sec:struct} 

The following combination of results of Baker~\cite[Theorem~3]{Bak0} and~\cite[Theorem~4]{Bak1} describes the structure 
of large Weyl sums.  
 
\begin{lemma} 
\label{lem:structure of large Weyl} We fix $d\ge 2$, some  $\varepsilon > 0$, and  suppose  that for a real   
$$
A> N^{1-1/D+ \varepsilon},
$$ 
where 
$$
D = \min\{2^{d-1}, 2d(d-1)\}, 
$$
we have $|\sfS_{d}(\vu; N)| \ge A$. 
Then there exist integers $q, r_1, \ldots, r_d$ such that 
$$
1 \le q \le  \(NA^{-1}\)^d N^{o(1)}, \qquad \gcd(q, r_1, \ldots, r_d) = 1, 
$$
and 
$$
\left| u_j - \frac{r_j}{q} \right | \le  q^{-1}  \(NA^{-1}\)^d N^{-j + o(1)}, 
\qquad j =1, \ldots, d. 
$$
\end{lemma} 

We now use Lemma~\ref{lem:structure of large Weyl} to get a slightly more precise statement.

\begin{lemma}
\label{lem:Struct Large Weyl}
We fix $d\ge 3$, some  $\varepsilon > 0$, and  suppose  that for a real 
$$
A> N^{1-1/D + \varepsilon}, 
$$ 
where 
$$
D = \min\{2^{d-1}, 2d(d-1)\}, 
$$
we have $|\sfS_{d}(\vu; N)| \ge A$. 
Then there exist positive integers $q_2 \ldots q_d$  with  $\gcd(q_i,q_j) = 1$, $2 \le i < j \le d$, such that  
\begin{itemize} 
\item[(i)] $q_2$ is  cube  free,
\item[(ii)]  $q_i$ is $i$-th power full but $(i+1)$-th power free when  $3\le i\le d-1$,
\item[(iii)]  $q_d$ is   $d$-th power full, 
\end{itemize}
and
$$
 \prod_{i=2}^d q_i^{1/i} \le N^{1+o(1)}A^{-1} 
 $$
and integers $b_1, \ldots, b_d$ 
with 
$$
\gcd\(q_2\ldots q_d, b_1, \ldots, b_d\)=1  
$$ 
such that 
$$
\left|u_j-\frac{b_j}{q_2\ldots q_d}\right |\le (NA^{-1})^{d} N^{-j+o(1)}  \prod_{i=2}^d q_i^{-d/i} , \qquad j=1, \ldots, d.
$$
\end{lemma}

\begin{proof}
By Lemma~\ref{lem:structure of large Weyl}, there exist integers $q, r_1, \ldots, r_d$ such that 
$$
1 \le q \le  \(NA^{-1}\)^d N^{o(1)} , \qquad \gcd(q, r_1, \ldots, r_d) = 1,
$$
and 
$$
\beta_j=\left| u_j - \frac{r_j}{q} \right | \le  q^{-1}  \(NA^{-1}\)^d N^{-j + o(1)}, 
\qquad j =1, \ldots, d. 
$$
By Lemma~\ref{lem:Weyl-Approx} we have 
\begin{equation}
\label{eq:Asymp S SI} 
\sfS_d(\vu; N)=q^{-1}S_{d,q}(\vb)\sfI(\bbeta)+\Delta,
\end{equation}
where $S_{d,q}(\vb)$ is given by~\eqref{eq:comp sum},   and with 
\begin{equation}
\label{eq:Delta} 
\Delta\ll  q+(NA^{-1})^{d} N^{o(1)}\le  (NA^{-1})^{d}N^{o(1)} .
\end{equation}
By the condition $A\ge N^{1-1/D+\varepsilon}$ and $d\ge 3$ we see from~\eqref{eq:Delta} that 
$|\Delta|\le A/2$ provided that $N$ is large enough. Thus, it follows from~\eqref{eq:Asymp S SI} and the triangle
inequality that 
\begin{equation}
\label{eq:Triang Ineq} 
A/2 \le |\sfS_{d}(\vu; N)| -  |\Delta|  \le  q^{-1}|S_{d,q}(\vb)| |\sfI(\bbeta)|. 
\end{equation}
We factorise  $q=q_2\ldots q_d$ as in  Lemma~\ref{lem:d-power-factor}. 
Hence, by Lemma~\ref{lem:d-power-factor} we have 
$$
|S_{d,q}(\vb)|\le \prod_{i=2}^d q_i^{1-1/i} q^{o(1)} .
$$
Thus, recalling the bound~\eqref{eq:Int(xi)}, we derive from~\eqref{eq:Triang Ineq}  
\begin{align*} 
A & \le N^{1+o(1)} \prod_{i=2}^d q_i^{1-1/i}   \min_{j=1, \ldots, d}\{1, \beta_j^{-1/d}N^{-j/d}\}\\
& = N^{1+o(1)}  \prod_{i=2}^d q_i^{-1/i}   \min_{j=1, \ldots, d}\{1, \beta_j^{-1/d}N^{-j/d}\}.
\end{align*} 
In particular, 
$$
A \le N^{1+o(1)}  \prod_{i=2}^d q_i^{-1/i},  
$$
which implies the desired restriction on $q_2, \ldots, q_d$.

Furthermore,  for each $j=1, \ldots, d$, we have 
$$
\beta_j\le (NA^{-1})^{d}N^{-j+o(1)} \prod_{i=2}^d q_i^{-d/i} ,
$$
which finishes the proof.
\end{proof}

\subsection{Frequency of large Weyl sums} 
Let $\lambda$ denote  the Lebesgue measure on $ \T_k$ (for an appropriate $k$). 

For $\vx \in \T_k$ we define  $\vy(\vx)$ by 
$$
 \left| S_{\bphi}( \vx, \vy(\vx); N)  \right| = 
\sup_{\vy \in \T_{d-k}} \left| S_{\bphi}( \vx, \vy; N)  \right|
$$
(if there are several choices we fix any, for example, the lexicographically smallest).

For any $A>0$ denote 
$$
\lambda_{\bphi, k} (A;N) =\lambda\(\{ \vx\in \T_k:~| S_{\bphi}( \vx, \vy(\vx); N) |\ge A\} \). 
$$

 We start with recalling the bound 
 $$\lambda_{\bphi, k} (A;N)
  \le N^{s(d)+\sigma_k(\bphi)+d-k +o(1)} A^{-2  s(d) - d +k } 
$$
from~\cite[Lemma~3.2]{ChSh1}, which using~\eqref{eq:sigmak} and~\eqref{eq:sd}
we write as follows.  

\begin{lemma} 
\label{lem:level set for all A}
Let $A$ be real number with $1 \le A \le N$.  Then 
$$\lambda_{\bphi, k} (A;N)
  \le N^{2s(d)+d-k -  \tau_k(\bphi)+o(1)} A^{-2  s(d) - d +k }. 
$$
\end{lemma}

We now show that sometimes we have  better bounds. 

For any integer $i\ge 2$ it is  convenient to denote   $$
\cF_{i}=\{n \in \N:~  \text{$n$ is $i$-th power full}\} \quad \text{and} \quad
\cF_{i}(x)= \cF_i\cap [1,x].
$$ 
The classical result of  Erd{\H o}s and  Szekeres~\cite{ErdSz} gives an asymptotic 
formula for the cardinality of $\cF_{i}(x)$ which we present here in a very relaxed form 
as the upper bound 
\begin{equation}
\label{eq:d-full}
\# \cF_{i}(x) \ll x^{1/i} . 
\end{equation}

\begin{lemma} 
\label{lem:level set for large A}
Suppose that $d\ge 3$ and$$
A> N^{1-1/D + \varepsilon}
$$ 
 for some fixed $\varepsilon>0$, where 
$$
D = \min\{2^{d-1}, 2d(d-1)\}.
$$
Then we have 
$$
\lambda_{\bphi, k} (A;N)\le  N^{dk+1-\tau_k(\bphi) + o(1)} A^{-dk-1}. 
$$
\end{lemma}

\begin{proof}
Denote  
$$
Q = (NA^{-1})^d.
$$
We also fix some $\eta> 0$. Let  $ \fU_{q_2, \ldots, q_d}$ be set of vectors $\vu\in \Tor$ with components  satisfying the inequalities of Lemma~\ref{lem:Struct Large Weyl}, that is,
\begin{align*} 
 \fU_{q_2, \ldots, q_d} &= \biggl\{\vu  = (u_1, \ldots, u_d) \in \Tor:\\
 & \qquad   \left|u_j-\frac{b_j}{q_2 \ldots q_d}\right |\le c Q N^{-\deg \varphi_j+\eta}  \prod_{i=2}^d q_i^{-d/i}, \
 j=1, \ldots, k\biggr\}
\end{align*} 
with some constant $c>0$, which depends only on $d$ and $\eta$.  
Clearly, 
\begin{equation} \label{eq:size Uqq}
\lambda(\fU_{q_2, \ldots, q_d}) \ll (q_2 \ldots q_d)^{k} Q^{k} N^{-\tau_k(\bphi)+k\eta}  \prod_{i=2}^d q_i^{-dk/i}. 
\end{equation}
 By Lemma~\ref{lem:Struct Large Weyl} we obtain 
 \begin{equation} \label{eq:Union Uqq}
\{ \vx\in \T_k:~|S_{\bphi} \(\vx, y(\vx); N\)|\ge A\} \subseteq \bigcup_{(q_2, \ldots, q_d) \in \Omega} 
\fU_{q_2, \ldots, q_d},
\end{equation} 
where, slightly relaxing the conditions of Lemma~\ref{lem:Struct Large Weyl}, for any $\eta> 0$ we can take
$$
\Omega =\left\{\(q_2, \ldots, q_d\)\in \N^{d-1}:~ q_i \in \cF_i, \ 3 \le i \le d, \  \prod_{i=2}^d q_i^{1/i} \le C Q^{1/d} N^{\eta}  \right\}
$$ 
for some constant $C>0$, which depends only on $d$ and $\eta$. 

For $Z \ge 1$, we  write $z \sim Z$ to denote that $Z/2 < z \le Z$. 
We now fix some real numbers $Q_2, \ldots, Q_d$ and consider the measure $U\(Q_2,  \ldots, Q_d\)$ 
of the contribution 
to  the right hand side of~\eqref{eq:Union Uqq} from 
$(q_2, \ldots, q_d) \in \Omega$ with $q_i \sim Q_i$, $2 \le i \le d$.  

Covering  $\Omega$ by $O\(\(\log N\)^d\)$ dyadic boxes, we see that from~\eqref{eq:Union Uqq} 
that 
 \begin{equation} \label{eq:Dyadic Uqq}
 \begin{split} 
\lambda_{\bphi, k} (A;N)
  \ll \max & \biggl \{U\(Q_2,  \ldots, Q_d\):~ Q_2, \ldots, Q_d \ge 1,\\
  &  \qquad \qquad \prod_{i=2}^d Q_i^{1/i} \le C Q^{1/d} N^{\eta}\biggr\} \(\log N\)^d .
  \end{split} 
\end{equation}

  Thus, using~\eqref{eq:size Uqq}, we obtain
\begin{align*} 
U\(Q_2,  \ldots, Q_d\) &\le  \sum_{\substack{(q_2, \ldots, q_d) \in \Omega\\ q_i \sim Q_i, \, 2 \le i \le d}} \lambda(\fU_{q_2, \ldots, q_d})\\ 
& \ll  Q^{k} N^{-\tau_k(\bphi)+k \eta}  \sum_{\substack{(q_2, \ldots, q_d) \in \Omega\\ q_i \sim Q_i, \, 2 \le i \le d}}  \prod_{i=2}^d q_i^{k-dk/i} .
\end{align*} 
Recalling the definition of $\Omega$, we see that 
\begin{equation} \label{eq:Bound UQQ 1}
\begin{split} 
 U\(Q_2,  \ldots, Q_d\) 
&  \ll  Q^{k} N^{-\tau_k(\bphi)+k\eta  }  \\
&\qquad \qquad \sum_{q_2\sim Q_2} q_2^{k-dk/2} 
 \prod_{i=3}^d  \sum_{\substack{ q_i \sim Q_i\\
q_i\in \cF_i\(Q_i\)} }  q_i^{k-dk/i} .
\end{split} 
\end{equation}
Applying the bound~\eqref{eq:d-full}, we derive from~\eqref{eq:Bound UQQ 1} that 
\begin{equation} \label{eq:Bound UQQ 2}
U\(Q_2,  \ldots, Q_d\) \ll  Q^{k} N^{-\tau_k(\bphi)+k \eta}  
 \prod_{i=2}^d   Q_i^{\alpha_i}, 
\end{equation}
 where
 $$
\alpha_2 =  k-dk/2 +1 \mand  \alpha_i = k-(dk-1)/i, \quad i = 3, \ldots,  d.
 $$
Observe that for  $i=2, \ldots, d$, we have 
\begin{equation} \label{eq:alpha_i}
 \alpha_i \le    1/i, 
\end{equation}
which for  $i \ge 3$ is obvious from 
 $$
 \alpha_i = k-(dk-1)/i = 1/i - (d/i -1)k
 $$
 and for $i=2$ from 
 $$
 \alpha_2 = 1 - (d/2-1) k
 $$
 and $d \ge 3$.

Since $Q_i \ge 1$,  using~\eqref{eq:alpha_i}, under the condition on $Q_{2}, \ldots, Q_d$ in~\eqref{eq:Dyadic Uqq}  we derive  
\begin{equation} \label{eq:Prod}
  \prod_{i=2}^d   Q_i^{\alpha_i} \le  
  \prod_{i=2}^d   Q_i^{1/i} \ll  Q^{1/d} N^{\eta},
\end{equation}
which is achieved for the choice $Q_{2}=\ldots = Q_{d-1}=1$ and $Q_d \sim Q N^{d\eta}$.

Combining~\eqref{eq:Prod} with~\eqref{eq:Dyadic Uqq} and~\eqref{eq:Bound UQQ 2}, 
we obtain  
  $$
\lambda_{\bphi, k} (A;N)\le  \(NA^{-1}\)^{dk+1} N^{-\tau_k(\bphi) + o(1)} 
$$
(since $\eta>0$ is arbitrary). This finishes the proof.
\end{proof}

\subsection{Quadratic Weyl sums}
In this case we have the following analogue of Lemma~\ref{lem:level set for large A}. 

\begin{lemma}
\label{lem:level sets-Y} 
For any $A\ge N^{1/2+\varepsilon}$ with  some fixed $\varepsilon>0$ we have 
$$
\lambda\(\left\{ y\in \T:~\sup_{x\in \T}|G(x, y; N)|\ge A\right \} \)\le
N^{2+o(1)}A^{-4}.
$$
\end{lemma}

\begin{proof}   
We   see from Lemma~\ref{lem:structure of large Weyl},    that there is some  $Q=(NA^{-1})^{2}N^{o(1)}$ such that 
if $|G(x, y; N)|\ge A$ for some $x \in T$, then  for some $q\le Q$ the coefficients belong to one of at most $q$ intervals
corresponding to $1 \le a_2 \le q$ and the length of each interval is $O\(QN^{-2} q^{-1} \)$.
Therefore
 for some   constant $c >0$ depending only on $\varepsilon$,  we have 
\begin{align*}
  \lambda\(\left\{ y\in \T:~\sup_{x\in \T}|G(x, y; N)|\ge A\right \} \) &
  \ll \sum_{1\le q \le  Q}  q \frac{Q}{qN^2}\\
  = QN^{-2}   \sum_{1\le q \le Q} 1 &
\le   Q^2N^{-2}= N^{2+o(1)}A^{-4}, 
\end{align*}
which concludes the proof. 
\end{proof}

We fix some $\varepsilon>0$. Denote $Q=(NA^{-1})^2N^{\varepsilon}$.
Suppose that 
$$
|G(x, y; N)|\ge A\ge N^{1/2+\varepsilon}.
$$
Then by Lemma~\ref{lem:structure of large Weyl} there exist $q\le Q$, $(a_1, a_2)\in [q]^2$ and constant $c>0$ which depends on $\varepsilon$ only such that 
$$
(x, y)\in \fR_{q,a_1, a_2},
$$
where the box
$$ 
\fR_{q,a_1, a_2} =  \left[\frac{a_1}{q}- c\frac{Q}{qN}, \frac{a_1}{q}+c\frac{Q}{qN}\right] \times
 \left[\frac{a_2}{q}- c\frac{Q}{qN^2}, \frac{a_2}{q}+c\frac{Q}{qN^2}\right].
$$

In what follows, for   $m\in \N$ it is convenient to denote 
$$[m]=\{0, 1, \ldots, m-1\}.
$$ 

We note that the implied constant below may depend on the parameter $t$.

\begin{lemma} \label{lem:level set for any t}
For any $t\in \R$ and any  $A\ge N^{1/2+\varepsilon}$ we have 
$$
 \lambda\(\left\{z\in \R:~\sup_{(x, y)\in \pi_{t}^{-1}(z)\cap \T_2}|G(x, y); N)|\ge A\right\}\)\le N^{5+o(1)}A^{-6}.
$$
\end{lemma}

\begin{proof} 
Applying Lemma~\ref{lem:structure of large Weyl} and elementary geometry, we have 
\begin{equation} \label{eq:subset}
\begin{split} 
&\left\{z\in \R:~\sup_{(x, y)\in \pi_{t}^{-1}(z)\cap \T_2}|G(x, y); N)|\ge A\right\}\\
&\qquad \qquad \qquad \qquad  \subseteq \pi_{t}\left (\bigcup_{q\leq Q} \bigcup_{(a_1, a_2)\in [q]^{2}} \fR_{q,a_1, a_2} \right )\\
&\qquad \qquad \qquad \qquad  =\bigcup_{q\le Q}  \pi_{t} \left (\bigcup_{(a_1, a_2)\in [q]^{2}} \fR_{q,a_1, a_2} \right ).
\end{split} 
\end{equation}
Observe that  for any integer $q \ge 1$ we have 
$$
\bigcup_{(a_1, a_2)\in [q]^{2}} \fR_{q,a_1, a_2}= \left\{(a_1/q, a_2/q):~(a_1, a_2)\in [q]^{2}\right\} +\fR_{q, 0, 0},  
$$
where  the summation symbol  `$+$' means  the {\it arithmetic (or Minkowski) sum\/}, that is, for sets $\cA, \cB\subseteq \R^d$,
$$
\cA+\cB=\{a+b:~a\in \cA, \ b\in \cB\}.
$$

Thus 
\begin{align*} 
 \pi_{t} \left (\bigcup_{(a_1, a_2)\in [q]^{2}} \fR_{q,a_1, a_2} \right )&=\pi_{t}(q^{-1}([q]\times[q]))+\pi_{t}(\fR_{q, 0, 0})\\
&=q^{-1}([q]+t[q])+\pi_{t}(\fR_{q, 0, 0}), 
\end{align*} 
where 
$$
[q]+t[q] = \{a_1 +  ta_2:~(a_1, a_2)\in [q]^{2}\}, 
$$
and for a set $\cS \subseteq \R$ and a scalar $\alpha \in \R$, we denote  $\alpha \cS = \{\alpha s:~s \in \cS\}$. 
Note that  
\begin{equation}
\label{eq:trivial}
\#([q]+t[q])\ll q^2. 
\end{equation}
It follows that 
\begin{align*}
\lambda\left(\pi_{t} \left (\bigcup_{(a_1, a_2)\in [q]^{2}} \fR_{q,a_1, a_2} \right ) \right)&\\
\le
 \#\(q^{-1}\([q]+t[q]\)\)& \lambda\(\pi_{t}(\fR_{q, 0, 0})\)  \ll qQ/N.
\end{align*}
Thus we see from~\eqref{eq:subset} that 
\begin{align*}
\lambda \left ( \left\{z\in \R:~\sup_{(x, y)\in \pi_{t}^{-1}(z)\cap \T_2}|G(x, y); N)|\ge A\right\} \right)&\\ 
  \ll Q^{3}/N &=N^{5+o(1)}A^{-6},
\end{align*}
which finishes the proof.
\end{proof}

We now turn our attention to the case  $t\in \Q$, and note that  for this setting the bound~\eqref{eq:trivial} can be 
replaced by  by $q$ (the smallest possible upper bound).  The geometric  meaning of this is that there are overlaps of the projection $\pi_t$ when $t\in \Q$. We record this as the following elementary sumset estimate.

\begin{lemma}
\label{lem:sumsets}
For $m\in \N$ denote $[m]=\{0, 1, 2, \ldots, m\}$. For $t\in \Q$ and $n\in \N$ we have 
$$
\#\([n]+t[n]\)\ll n.
$$
\end{lemma}

\begin{proof}
Suppose that $t=a/b$ with $a, b\in \Z$ and $b\neq 0$. Then we have 
$$
\#\([n]+t[n]\)=\#\(b[n]+a[n]\)  \leq (a+b)n+1,
$$
which finishes the proof.
\end{proof}

\begin{lemma}
\label{lem:level-sets-rational-t}
For any $t\in \Q$ and any $A\ge N^{1/2+\varepsilon}$ we have 
$$
 \lambda\(\left\{z\in \R:~\sup_{(x, y)\in \pi_{t}^{-1}(z)}|G(x, y); N)|\ge A\right\}\) \le N^{3+o(1)}A^{-4}.
$$
\end{lemma}

\begin{proof}
Applying the similar arguments as in the proof of Lemma~\ref{lem:level set for any t}, but using the bound
$$
\#\([n]+t[n]\)\ll q
$$ 
of  Lemma~\ref{lem:sumsets}, 
instead of~\eqref{eq:trivial}, we obtain the desired bound.  
\end{proof}

\section{Proofs of main results}

\subsection{Preliminaries} 
We recall the following rather elementary but useful general result, which is a slightly modified version of \cite[Lemma 4.1]{ChSh1}.

\begin{lemma} 
\label{lem:general lemma} 
Let $\cX$ be a metric space and $\nu$ be a Radon measure on $\cX$ with $\nu(\cX)<\infty$. 
Let $M\le N$ be two  positive numbers and $F:\cX\rightarrow [0, N]$ be a function such that for any $M\le A\le N$,
$$
\nu(\{\vx\in \cX:~F(\vx)\ge A\})\le N^{a}A^{-b}.
$$
Then for any $\rho>0$,
$$
\int_{\cX} F(\vx)^{\rho}d\nu(\vx)\ll \nu(\cX) M^{\rho}+N^{a}M^{\rho-b}\log N+N^{\rho+a-b}\log N.
$$
\end{lemma} 

\begin{proof}
Taking a dyadic decomposition, we obtain 
$$
\{\vx\in \T_k:~M< F(\vx)\le N\}=\bigcup_{i=1}^{I} \{\vx\in \T_k:~M2^{i-1}< F(\vx)\le 2^{i}M\},
$$
where $I$ is the integer number such that $2^{I-1}M\le N <2^{I}M$. 
Thus we obtain 
\begin{align*}
& \int_{\cX}F(\vx)^{\rho} d\nu(\vx)=\int_{\{\vx: F(\vx)\le M\}} F(\vx) d\nu(\vx)+ \int_{\{\vx: M<F(\vx)\le N\}} F(\vx) d\nu(\vx)\\
&\qquad \qquad  \le M^{\rho}+ \sum_{i=1}^{I} (2^{i}M)^{\rho}\nu(\{\vx\in \T_k:~M2^{i-1}< F(\vx)\le 2^{i}M\})\\
&\qquad \qquad  \ll M^{\rho}+N^{a}M^{\rho-b}\sum_{i=1}^{I}  2^{i(\rho-b)}.
\end{align*}
Considering the cases $\rho\le b$ or $\rho>b$ separately, we obtain the desired bounds. 
\end{proof}

\subsection{Proof of Theorem~\ref{thm:general-large rho}}

Let $a=s(d)+\sigma_k+d-k$, $b=2s(d)+d-k$ and $M=N^{a/b}$.  Applying Lemmas~\ref{lem:level set for all A} and~\ref{lem:general lemma} we obtain 
$$
\int_{\T_k} |S_{\bphi}(\vx, y(\vx);N )|^{\rho} d\vx\le N^{\rho a/b+o(1)}+N^{\rho+a-b+o(1)}\le N^{\rho+a-b+o(1)}.
$$
The second inequality holds since $\rho\ge b$. By our notation, we have 
$$
\rho+a-b=\rho-\tau_k(\bphi),
$$
which gives the desired upper bound.

For the lower bound we note that for an appropriate constant $c >0$,  which depends only on $d$, if 
$$
0 \le x_i  \le c N^{-\deg \varphi_i}, \qquad i =1, \ldots, k, 
$$
and  $\vy = {\mathbf 0}$ we have $|S_{\bphi}(\vx, {\mathbf 0};N)| \gg N$.
Hence, 
\begin{align*}
\int_{\T_k} & \sup_{\vy \in\T_{d-k}} \left| S_{\bphi}( \vx, \vy; N)  \right|^{\rho}d\vx \\
& \qquad \ge \int_0^{ c N^{-\deg \varphi_1}} \ldots  \int_0^{ c N^{-\deg \varphi_k}}  
\sup_{\vy \in\T_{d-k}} \left| S_{\bphi}( \vx, \vy; N)  \right|^{\rho}  dx_1\ldots d x_k\ \\
&  \qquad\ge \int_0^{ c N^{-\deg \varphi_1}} \ldots  \int_0^{ c N^{-\deg \varphi_k}}  \left| S_{\bphi}( \vx,  {\mathbf 0}; N)  \right|^{\rho} \gg
 N^{\rho-\tau_k(\bphi)}, 
\end{align*}
which concludes the proof.

\subsection{Proof of Theorem \ref{thm:smaller rho}} 
Let fix some $\varepsilon > 0$ and set 
$$
a=dk+1-\tau_k(\bphi), \qquad  b=dk+1, \qquad M=N^{1-1/2^{d-1}+\varepsilon}.
$$ 
We assume that  $\varepsilon > 0$ is sufficiently small, so that $M < N$.

Applying Lemmas~\ref{lem:level set for large A} and~\ref{lem:general lemma} for 
$$
\cI =  \int_{\T_k} |S_{\bphi}(\vx, y(\vx);N )|^{\rho} d\vx 
$$
we obtain 
$$
\cI  \ll M^{\rho}+N^{a+o(1) }M^{\rho-b} +N^{\rho+a-b+o(1)}.
$$

{\bf Case 1.} Suppose that $0<\rho\le b$.
Then $M^{\rho-b}\ge N^{\rho+a-b}$ and we see that 
$$
 \cI \ll M^{\rho}+N^{a+o(1)}M^{\rho-b}.
$$
Taking the values $M, a, b$ and using the fact  that for any $d\ge 3$,
$$
\tau_k(\bphi)\ge \frac{dk+1}{2^{d-1}},
$$ 
we obtain $M^{\rho}\ge N^{a}M^{\rho-b}$, and hence 
$$
\cI \ll M^{\rho}=N^{\rho(1-1/2^{d-1}+\varepsilon)}.
$$

{\bf Case 2.} Suppose that $\rho>b$. Then $N^aM^{\rho-b} < N^{\rho+a-b}$ and we see that 
$$
 \cI \ll M^{\rho} +N^{\rho+a-b+o(1)}.
$$

(i) If $b<\rho\le \tau_k(\bphi)2^{d-1}$ then $M^{\rho}\ge N^{\rho+a-b}$, and hence 
$$
\cI \ll N^{\rho (1-1/2^{d-1}+\varepsilon+o(1))}.
$$

(ii) If $\rho\ge \tau_k(\bphi) 2^{d-1}$ then we have 
$$
\cI \ll N^{\rho+a-b+\varepsilon + o(1)}=N^{\rho-\tau_k(\bphi)+\varepsilon + o(1)}.
$$

Putting all cases together, since $\varepsilon>0$ is arbitrary, we obtain the desired result. 

\subsection{Proof of Theorem \ref{thm:Gauss L rho}}

Let $a=2+\varepsilon$, $b=4$ and $M= N^{1/2+\varepsilon}$. By Lemma \ref{lem:level sets-Y} and Lemma \ref{lem:general lemma} we obtain 
$$
\int_{\T} \sup_{x\in \T}|  G(x, y; N)| dy\ll N^{\rho/2+\varepsilon}+N^{\rho-2+\varepsilon},
$$
since $\varepsilon>0$ is arbitrary, which finishes the proof.

\subsection{Proof of Theorem \ref{thm:any-t}}

Let $a=5+\varepsilon$, $b=6$ and $M\ge N^{1/2+\varepsilon}$ be some parameter which is to be determined later. By Lemmas~\ref{lem:level set for any t} and~\ref{lem:general lemma} we obtain 
$$
\int_{\R}  \sup_{(x, y)\in \pi_{t}^{-1}(z)\cap \T_2 } \left| G(x,y; N) \right|^\rho d z\ll M^{\rho}+N^{5}M^{\rho-6}+N^{\rho-1}.
$$
Hence, taking $M=N^{5/6}$, since $\varepsilon>0$ is arbitrary, we obtain the desired bound. 

\subsection{Proof of Theorem \ref{thm:rational-t}} Taking a similar argument as in the proof of Theorem \ref{thm:any-t} 
and using Lemma~\ref{lem:level-sets-rational-t}  instead of Lemma~\ref{lem:level set for any t}, with  $M=N^{3/4}$ we obtain the desired bound.

\section{Comments}

Perhaps one can use the description of large sums in Section~\ref{sec:struct}   
to get news estimates for the mean values
$$
\int_{\Gamma}  \left| \sfS_{d}(\vu; N)  \right|^{\rho}d\mu(\vu)
\mand 
\int_{\fB}  \left| \sfS_{d}(\vu; N)  \right|^{\rho}d\vu
$$
for a smooth surface $\Gamma$ with an attached Radon measure $\mu$ on $\Gamma$ 
and for a box $\fB= [\xi_1, \xi_1+\delta] \times \ldots \times  [\xi_d, \xi_d+\delta]$ for 
a small $\delta$. Several results for such average values can be found in~\cite{CKMS, ChSh3,DeLa}. 
We hope that our approach can improve them in some ranges.

\section*{Acknowledgement} 

 The authors are grateful to Lillian Pierce for stimulating conversations which have 
 led to this work. 

This work of I.S was  supported   by ARC Grant~DP170100786.

 \end{document}